\documentclass[12pt,reqno]{amsart}

\usepackage{hyperref}
\usepackage{amsthm}
\usepackage{amssymb}
\usepackage{amsxtra}
\usepackage{kopp} 
\usepackage{epsfig}
\usepackage{verbatim}
\usepackage{cleveref}

\numberwithin{equation}{section}

\theoremstyle{plain}
\newtheorem{theorem}{Theorem}[section]
\newtheorem{thm}[theorem]{Theorem}
\newtheorem{prop}[theorem]{Proposition}
\newtheorem{cor}[theorem]{Corollary}
\newtheorem{lem}[theorem]{Lemma}

\newtheorem{conj}[theorem]{Conjecture}

\theoremstyle{definition}
\newtheorem{defn}[theorem]{Definition}

\theoremstyle{remark}

\include{header}

\begin{document}


\title[A Kronecker limit formula for indefinite zeta functions]{A Kronecker limit formula for indefinite zeta functions}
\author{Gene S. Kopp}

\address{School of Mathematics, University of Bristol, Bristol, UK and Heilbronn Institute for Mathematical Research, Bristol, UK}
\email{gene.kopp@bristol.ac.uk}


\date{\today}

\begin{abstract} 
We prove an analogue of Kronecker's second limit formula for a continuous family of ``indefinite zeta functions''. Indefinite zeta functions were introduced in the author's previous paper as Mellin transforms of indefinite theta functions, as defined by Zwegers. Our formula is valid in dimension $g=2$ at $s=1$ or $s=0$. For a choice of parameters obeying a certain symmetry, an indefinite zeta function is a differenced ray class zeta function of a real quadratic field, and its special value at $s=0$ was conjectured by Stark to be a logarithm of an algebraic unit. Our formula also permits practical high-precision computation of Stark ray class invariants.
\end{abstract}

\maketitle

\section{Introduction}

In a previous paper \cite{kopp1}, we introduced indefinite zeta functions as Mellin transforms of certain indefinite theta functions associated to the intermediate Siegel half-space $\HH_g^{(1)}$, defined below. 
In this paper, we obtain a formula for the values of such an indefinite zeta function at $s=1$ or $s=0$, in 
the special case of dimension $g=2$. Such formulas are traditionally called Kronecker limit formulas, after Kronecker's first and second limit formulas giving the constant term in the Laurent expansion at $s=1$ of standard and twisted real analytic Eisenstein series. 

When our parameters are specialised appropriately, our special value is a finite linear combination of Hecke $L$-values at $s=1$. Our formula may be used to compute values of Hecke $L$-functions at $s=1$ (resp. $s=0$) relevant to the Stark conjectures, which we discuss in \Cref{subsec:15}.

For imaginary quadratic fields, Stark proved his conjectures using Kronecker's first and second limit formulas together with the theory of singular moduli \cite{stark1}. The Kronecker limit formulas give the constant Laurent series coefficient at $s=1$ for families of Dirichlet series continuously interpolating the ray class zeta functions $\zeta(s,A)$---namely, standard and twisted real analytic Eisenstein series (see \cite{tata} for details). We will discuss a generalisation of these Kronecker limit formulas, which correspond to the positive definite case, in \cite{kopp2}.

Kronecker limit formulas applicable to real quadratic fields were developed by Hecke, Herglotz, 
and Zagier (in analogy with the first Kronecker limit formula), and by Shintani (in analogy with the second Kronecker limit formula). As in the imaginary quadratic case, these formulas are obtained by continuously interpolating between ray class zeta functions using a larger family of functions. 
 Hecke's formula uses cycle integrals of real analytic Eisenstein series, whereas the formulas of Herglotz \cite{herglotz} and Zagier \cite{zagier} (see also \cite{duke2,radchenko}) and the formulas of Shintani \cite{shintanik,shintani,shintanicertain} use partial zeta functions defined by summing over a cone.
Analogues of the Kronecker limit formulas in other settings have been found by Liu and Masri \cite{liu}, Posingies \cite{posingies}, and Vlasenko and Zagier \cite{higher}, among others.

The main theorem of this paper supplies a new real quadratic analogue of Kronecker's second limit formula 
based on a new interpolation between ray class zeta functions. The interpolation is by the indefinite zeta functions introduced in \cite{kopp1}, which are Mellin transforms of nonholomorphic indefinite theta functions. Indefinite zeta functions have a nice functional equation, but they do not have a Dirichlet series representation for general parameters. 

The main results on indefinite zeta functions---stated in \Cref{sec:klfint}---require a lot of notation, defined in \Cref{sec:notation} and \Cref{sec:indefdef}. The proofs of the indefinite Kronecker limit formulas are provided in \Cref{sec:indefinite}.

\subsection{Notational conventions}\label{sec:notation}

We list some notational conventions used in the paper.
\begin{itemize}
\item $e(z) := \exp(2\pi i z)$ is the complex exponential, and this notation is used for $z \in \C$ not necessarily real.
\item $\HH := \{\tau : \im{\tau}>0\}$ is the complex upper half-plane.
\item Non-transposed vectors $v \in \C^g$ are always column vectors; the transpose $v^\top$ is a row vector.
\item If $M$ is a $g \times g$ matrix, then $M^\top$ is its transpose, and (when $M$ is invertible) $M^{-\top}$ is a shorthand for $\left(M^{-1}\right)^\top$.
\item $Q_M(v)$ denotes the quadratic form $Q_M(v) := \foh v^\top M v$, where $M$ is a $g \times g$ matrix, and $v$ is a $g \times 1$ column vector.
\item $\left.f(c)\right|_{c=c_1}^{c_2} := f(c_2)-f(c_1)$, where $f$ is any function taking values in an additive group.
\item If $v = \smcoltwo{v_1}{v_2} \in \C^2$ and $f$ is a function of $\C^2$, we may write $f(v)$ as $f\!\smcoltwo{v_1}{v_2}$ rather than $f\!\left(\smcoltwo{v_1}{v_2}\right)$.
\item We often express $\Omega = i M + N$ where $M, N$ are real $g \times g$ symmetric matrices; $N$ and $M$ will always have real entries even when we do not say so explicitly.
\end{itemize}

We use complex logarithms throughout this paper. If $f(\t)$ is any nonvanishing holomorphic function on the upper half plane $\HH$, there is some holomorphic function $(\Log f)(\t)$ such that $\exp\left((\Log f)(\t)\right) = f(\t)$, because $\HH$ is simply connected. Specifying a single value (or the limit as $\t$ approaches some element of $\R \cup \{\infty\}$) specifies $\Log f$ uniquely. It won't necessarily be true that $(\Log f)(\t) = \log(f(\t))$. 

Conventions for square roots, when not specified, follow \cite{kopp1}. See Section 2.3 and Section 3.2 therein for details.

We recall the definition of the Siegel intermediate half-space, as defined in \cite{kopp1}.
\begin{defn}
For $0 \leq k \leq g$, we define the \textbf{Siegel intermediate half-space} of genus $g$ and index $k$ to be
\begin{equation}
\HH_g^{(k)} := \{\Omega \in \M_g(\C) : \Omega = \Omega^\top \mbox{ and } \im(\Omega) \mbox{ has signature } (g-k,k)\}.
\end{equation}
\end{defn}
The $\HH_g^{(k)}$ are the open orbits of the action of $\Sp_{2g}(\R)$ by fractional linear transformations on the space of complex symmetric matrices. In particular, $\HH_g^{(0)}$ is the usual Siegel upper half-space.

\subsection{Indefinite theta and zeta functions}\label{sec:indefdef}

We review the relevant definitions from \cite{kopp1}.

\begin{defn}
For any complex number $\alpha$, define the function
\begin{equation}
\eE(\alpha) := \int_{0}^\alpha e^{-\pi u^2}\,du,
\end{equation}
where the integral runs along any contour from $0$ to $\alpha$.
\end{defn}

\begin{defn}
Let $\Omega=iM+N$ be a complex symmetric matrix whose imaginary part has signature $(g-1,1)$; that is, $\Omega \in \HH_g^{(1)}$. Define the (nonholomorphic) \textbf{indefinite theta function}
\begin{equation}\label{eq:theta1}
\Theta^{c_1,c_2}(z,\Omega) := \sum_{n \in \Z^g} \left.\eE\!\left(\frac{c^\top\im(\Omega n+z)}{\sqrt{-\foh c^\top \im(\Omega) c}}\right)\right|_{c=c_1}^{c_2} e\!\left(\frac{1}{2}n^\top\Omega n + n^\top z\right),
\end{equation}
where $z \in \C^g$, $c_1, c_2 \in \C^g$, 
$\ol{c_1}^\top M c_1 < 0$, and $\ol{c_2}^\top M c_2 < 0$.
\end{defn}

Nonholomorphic indefinite theta functions were first studied by Vign\'{e}ras \cite{vig1,vig2} and were rediscovered by Zwegers \cite{zwegers}.
Zwegers's theta function is defined for real $c_j$ when $N$ is a scalar multiple of $M$. 
More precisely, if $M$ is real symmetric matrix of signature $(g-1,1)$, $\tau \in \HH$, and $c_1,c_2 \in \R^g$,
then $\Theta^{c_1,c_2}(Mz,\tau M)$
is equal up to an exponential factor to the function $\vartheta^{c_1,c_2}_M(z,\tau)$ introduced by Zwegers on page 27 of \cite{zwegers}. Our theta functions extend Zwegers's to the Siegel modular setting; a related generalisation has also been studied by Roehrig \cite{roehrig}.

\begin{defn}\label{def:indefchar}
Let $\Omega=iM+N \in \HH_g^{(1)}$. 
Define the \textbf{indefinite theta null with characteristics $p,q \in \R^g$}:
\begin{align}
\Theta^{c_1,c_2}_{p,q}(\Omega) &:= e\!\left(\foh q^\top\Omega q + p^\top q\right) \Theta^{c_1,c_2}\!\left(p+\Omega q; \Omega\right).
\end{align}
where $c_1, c_2 \in \C^g$, 
$\ol{c_1}^\top M c_1 < 0$, and $\ol{c_2}^\top M c_2 < 0$.
\end{defn}

We define the indefinite zeta function using a Mellin transform of the indefinite theta function with characteristics.

\begin{defn}
Let $\Omega=iM+N \in \HH^{(1)}_g$. 
The \textbf{completed indefinite zeta function} is
\begin{equation}
\widehat{\zeta}^{c_1,c_2}_{p,q}(\Omega,s) := \int_0^\infty \Theta^{c_1,c_2}_{p,q}(t\Omega)t^s\frac{dt}{t}, \label{eq:indefzdef1}
\end{equation}
where $p,q \in \R^g$, and $c_1, c_2 \in \C^g$ are parameters satisfying
$\ol{c_1}^\top M c_1 < 0$ and $\ol{c_2}^\top M c_2 < 0$.
\end{defn}

The completed indefinite zeta function has an analytic continuation and satisfies a functional equation, which is Theorem 1.1 of \cite{kopp1}.

\begin{thm}[Analytic continuation and functional equation for $\widehat{\zeta}^{c_1,c_2}_{p,q}(\Omega,s)$]\label{thm:zetafun}
The function $\widehat{\zeta}^{c_1,c_2}_{p,q}(\Omega,s)$ may be analytically continued to an entire function on $\C$.
It satisfies the functional equation 
\begin{equation}
\widehat{\zeta}^{c_1,c_2}_{p,q}\!\left(\Omega,\frac{g}{2}-s\right) = \frac{e(p^\top q)}{\sqrt{\det(-i\Omega)}} \widehat{\zeta}^{\ol\Omega c_1,\ol \Omega c_2}_{-q,p}\!\left(-\Omega^{-1},s\right).
\end{equation}
\end{thm}

\subsection{Kronecker limit formulas for indefinite zeta functions}\label{sec:klfint}

The Kronecker limit formula for indefinite zeta functions involves the dilogarithm function and a rapidly convergent integral of a logarithm of an infinite product. We also require the following definition of the function $\kappa^c_{\Omega}(v)$, which is the square root of a rational function and will appear as a factor in the integrand.

\begin{defn}\label{defn:kappa}
Suppose $\Omega = iM+N \in \HH_2^{(1)}$, $c \in \C^2$ satisfying $\ol{c}^\top M c < 0$, $v \in \C^2$, and $s \in \C$. 
Let $\Lambda_\Omega^c := \Omega - \frac{i}{Q_M(c)}Mcc^\top M$. 
Then, we define
\begin{equation}
\k_{\Omega}^{c}(v) := \frac{c^\top Mv}{4\pi i \sqrt{-Q_M(c)} Q_\Omega(v)\sqrt{-2iQ_{\Lambda_\Omega^c}(v)}}.
\end{equation}
\end{defn}
We now state the formula.

\begin{thm}[Indefinite Kronecker limit formula at $s=1$]\label{thm:KLF30}
Let $\Omega = iM+N \in \HH_2^{(1)}$, $p = \smcoltwo{p_1}{p_2} \in \R^2 \setminus \Z^2$, and $c_1,c_2 \in \C^2$ such that $\ol{c_j}^\top M c_j < 0$. 
For $c = c_1,c_2$, factor the quadratic form 
\begin{equation}
Q_{\Lambda_\Omega^c}\!\coltwo{\xi}{1} = \alpha(c)(\xi-\t^+(c))(\xi-\t^-(c)),
\end{equation}
where $\t^+(c)$ is in the upper half-plane and $\t^-(c)$ is in the lower half-plane.
Then, 
\begin{align}
\widehat{\zeta}^{c_1,c_2}_{p,0}(\Omega,1) &= 
I^+(c_2) - I^-(c_2) - I^+(c_1) + I^-(c_1),
\end{align}
where
\begin{align}
I^{\pm}(c) &:= 
-\Li_2(e(\pm p_1))\k^{c}_{\Omega}\!\coltwo{1}{0} \nn \\
&\ \ \ +2i\int_0^\infty \left(\Log \varphi_{p_1,\pm p_2}\right)(\pm\tau^{\pm}(c)+it) \k_{\Omega}^{c}\!\coltwo{\pm\left(\tau^{\pm}(c)+it\right)}{1}\,dt.
\end{align}
The function $\varphi_{p_1,p_2} : \HH \to \C$ is defined by the a product expansion,
\begin{equation}
\varphi_{p_1,p_2}(\xi) :=
(1-e(p_1\xi_t+p_2))\prod_{d=1}^\infty\frac{1-e\left((d+p_1)\xi+p_2\right)}{1-e\left((d-p_1)\xi-p_2\right)},
\end{equation}
and its logarithm $\left(\Log \varphi_{p_1,p_2}\right)(\xi)$ is the unique continuous branch with the property
\begin{equation}
\lim_{\xi \to i\infty} \left(\Log\varphi_{p_1,p_2}\right)(\xi) = \left\{\begin{array}{ll}
\log(1-e(p_2)) & \mbox{ if } p_1 = 0, \\
0 & \mbox{ if } p_1 \neq 0.
\end{array}\right.
\end{equation}
Here $\log(1-e(p_2))$ is the standard principal branch.
\end{thm}
The following specialisation looks somewhat simpler and contains all of the cases of arithmetic zeta functions $Z_A(s)$ associated to real quadratic fields.

\begin{thm}[Indefinite Kronecker limit formula at $s=1$, pure imaginary case]\label{thm:KLF3}
Let $M$ be a $2 \times 2$ real matrix of signature $(1,1)$, and let $\Omega=iM$. Let $p = \smcoltwo{p_1}{p_2} \in \R^2\setminus \Z^2$, and $c_1,c_2 \in \R^2$ such that $c_j^\top M c_j < 0$. Then,
\begin{align}
\widehat{\zeta}^{c_1,c_2}_{p,0}(\Omega,1) &= 
2i\im\left(I(c_2) - I(c_1)\right),
\end{align}
where
\begin{align}
I(c) &= 
-\Li_2(e(p_1))\k^{c}_{\Omega}\!\coltwo{1}{0} \nn \\
&\ \ \ +2i\int_0^\infty \left(\Log \varphi_{p_1,p_2}\right)(\tau(c)+it) \k_{\Omega}^{c}\!\coltwo{\tau(c)+it}{1}\,dt.
\end{align}
Here, $\Log \varphi_{p_1,p_2}$ and $\k_{\Omega}^{c}$ are defined as in the statement of \Cref{thm:KLF30}, and 
$\xi=\t(c)$ is the unique root of the quadratic polynomial $Q_{\Lambda_\Omega^c}\!\smcoltwo{\xi}{1}$ in the upper half plane.
\end{thm}

It is straightforward to use the functional equation for the indefinite zeta function to rephrase \Cref{thm:KLF30} and \Cref{thm:KLF3} as formulas for $\widehat{\zeta}_{0,q}^{c_1,c_2}(\Omega,0)$. 

\begin{thm}[Indefinite Kronecker limit formula at $s=0$]\label{thm:KLF40}
Let $\Omega = iM+N \in \HH_2^{(1)}$, $q = \smcoltwo{q_1}{q_2} \in \R^2 \setminus \Z^2$, and $c_1,c_2 \in \C^2$ such that $\ol{c_j}^\top M c_j < 0$. 
For $c = c_1,c_2$, factor the quadratic form 
\begin{equation}
Q_{\Lambda_{-\Omega^{-1}}^{\ol{\Omega}c}}\!\coltwo{\xi}{1} = \beta(c)(\xi-\omega^+(c))(\xi-\omega^-(c)),
\end{equation}
where $\omega^+(c)$ is in the upper half-plane and $\omega^-(c)$ is in the lower half-plane.
Then, 
\begin{align}
\widehat{\zeta}^{c_1,c_2}_{0,q}(\Omega,1) &= 
\frac{1}{\sqrt{\det(-i\Omega)}}\left(J^+(c_2) - J^-(c_2) - J^+(c_1) + J^-(c_1)\right),
\end{align}
where
\begin{align}
J^{\pm}(c) &:= 
-\Li_2(e(\mp q_1))\k^{\ol{\Omega}c}_{-\Omega^{-1}}\!\coltwo{1}{0} \nn \\
&\ \ \ +2i\int_0^\infty \left(\Log \varphi_{-q_1,\mp q_2}\right)(\pm\omega^{\pm}(c)+it) \k_{-\Omega^{-1}}^{\ol{\Omega}c}\!\coltwo{\pm\left(\omega^{\pm}(c)+it\right)}{1}\,dt.
\end{align}
Here, $\Log \varphi$ and $\k$ are defined as in the statement of \Cref{thm:KLF30}.
\end{thm}

\begin{thm}[Indefinite Kronecker limit formula at $s=0$, pure imaginary case]\label{thm:KLF4}
Let $M$ be a $2 \times 2$ real matrix of signature $(1,1)$, and let $\Omega=iM$. Let $q = \smcoltwo{q_1}{q_2} \in \R^2\setminus \Z^2$, and $c_1,c_2 \in \R^2$ such that $c_j^\top M c_j < 0$. Then,
\begin{align}
\widehat{\zeta}^{c_1,c_2}_{0,q}(\Omega,0) &= 
\frac{2i}{\sqrt{\det(M)}}\im\left(J(c_2) - J(c_1)\right),
\end{align}
where
\begin{align}
J(c) &= 
-\Li_2(e(-q_1))\k^{\ol{\Omega}c}_{-\Omega^{-1}}\coltwo{1}{0} \nn \\
&\ \ \ +2i\int_0^\infty \left(\Log \varphi_{-q_1,-q_2}\right)(\omega(c)+it) \k_{-\Omega^{-1}}^{\ol{\Omega}c}\coltwo{\omega(c)+it}{1}\,dt.\label{eq:klf4int}
\end{align}
Here, $\Log \varphi$ and $\k$ are defined as in the statement of \Cref{thm:KLF30}, and 
$\xi=\omega(c)$ is the unique root of the quadratic polynomial $Q_{\Lambda_{-\Omega^{-1}}^{\ol{\Omega}c}}\!\smcoltwo{\xi}{1}$ in the upper half plane.
\end{thm}

\subsection{Application: indefinite zeta functions, real quadratic fields, and Stark units}\label{subsec:15}

The Hecke $L$-value $L_K(1, \chi)$ contains arithmetic information that is not well-understood in general. The abelian Stark conjectures predict that this value is an algebraic number times a regulator $\Reg_\chi$, which is a determinant of a matrix of linear forms in logarithms of algebraic units in a particular abelian extension of the number field $K$ \cite{stark1,stark2,stark3,stark4}. This conjecture is known when the base field $K$ is equal to $\Q$ or an imaginary quadratic field, but not (for instance) when $K$ is a real quadratic field.

The rank 1 abelian Stark conjectures give a partial answer to Hilbert's 12th Problem, which asked for explicit generators for the abelian extensions of a number field in terms of special values of transcendental functions. Computationally, the Stark conjectures are used to calculate class fields in the computer algebra systems Magma and PARI/GP.

The Stark conjectures are most precisely formulated in the rank 1 case---that is, when $L_K(s,\chi)$ vanishes to order $1$ at $s=0$. The regulator $\Reg_\chi$ in that case is a determinant of a $1 \times 1$ matrix. The Stark conjectures are most succinctly written as a statement about the ray class zeta function (of a ray ideal class $A$)
\begin{equation}
\zeta(s,A) := \zeta_K(s,A) := \sum_{\aa \in A} N(A)^{-s}
\end{equation}
rather that as a statement about the Hecke $L$-function
\begin{equation}
L_K(s,\chi) = \sum_A \chi(A) \zeta(s,A).
\end{equation}

Just as definite zeta functions specialise to ray class zeta functions of imaginary quadratic fields, indefinite zeta functions specialise to \textit{differenced ray class zeta functions} of real quadratic fields. The full details of this specialisation are given in Section 7 of \cite{kopp1}.

\begin{defn}[Ray class zeta function]\label{defn:rayzeta}
Let $K$ be any number field and $\cc$ an ideal of the maximal order $\OO_K$. Let $S$ be a subset of the real places of $K$ (i.e., the embeddings $K \inj \R$). Let $A$ be a ray ideal class modulo $\cc S$, that is, an element of the group
\begin{equation}\label{eq:Cl}
\Cl_{\cc S}(\OO_K) := \frac{\{\mbox{nonzero fractional ideals of $\OO_K$ coprime to $\cc$}\}}{\{a\OO_K : a \con 1 \Mod{\cc} \mbox{ and $a$ is positive at each place in $S$}\}}.
\end{equation}
For $\re(s)>1$, define the \textbf{ray class zeta function of $A$} to be
\begin{equation}
\zeta(s,A) := \sum_{\aa \in A} N(\aa)^{-s}.
\end{equation}
\end{defn}
This function has a simple pole at $s=1$ with residue independent of $A$. The pole may be eliminated by considering the function $Z_A(s)$, defined as follows.

\begin{defn}[Differenced ray class zeta function]\label{defn:diffzeta}
Let $R$ be the element of $\Cl_{\cc S}(\OO_K)$ defined by 
\begin{equation}
R:= \{a\OO_K : a \con -1 \Mod{\cc} \mbox{ and $a$ is positive at each place in $S$}\}.
\end{equation}
For $\re(s)>1$, define the \textbf{differenced ray class zeta function of $A$} to be
\begin{equation}
Z_A(s) := \zeta(s,A) - \zeta(s,RA).
\end{equation}
\end{defn}
The function $Z_A(s)$ extends to a holomorphic function on the whole complex plane. The rank 1 abelian Stark conjecture says that $Z_A'(0)$ is the logarithm of an algebraic unit.

\begin{conj}[Stark \cite{stark3}]\label{conj:stark}
Let $K$ be a real quadratic field and $\{\rho_1, \rho_2\}$ the real embeddings of $K$. If $R$ is not the identity of $\Cl_{\cc\infty_2}(\OO_K)$, then $Z_A'(0) = \log(\rho_1(\e_A))$ for an algebraic unit $\e_A$ generating the ray class field $L_{\cc\infty_2}$ corresponding to $\Cl_{\cc\infty_2}(\OO_K)$.
The units are compatible with the Artin map: $\e_{\id }^{\Art(A)} = \e_A$.
\end{conj}

The specialisation of the indefinite zeta function to a differenced real quadratic zeta function is given by the following result, which is Theorem 1.3 of \cite{kopp1}.

\begin{thm}[Specialisation of indefinite zeta function at $s=0$] 
\label{thm:special}
For each ray class $A \in \Cl_{\cc\infty_1\infty_2}(\OO_K)$ and integral ideal $\bb \in A^{-1}$, there exists a real symmetric $2 \times 2$ matrix $M$, vectors $c_1, c_2 \in \R^2$, and $q \in \Q^2$ such that
\begin{equation}\label{eq:special}
(2\pi N(\bb))^{-s}\Gamma(s)Z_A(s) = \widehat{\zeta}^{c_1,c_2}_{0,q}(iM,s).
\end{equation}
\end{thm}
We may use \Cref{thm:special} to compute presumptive Stark units $\exp(Z_A'(0))$. Specifically,

\begin{cor}\label{cor:special}
Under the specialisation given by \Cref{thm:special},
\begin{equation}
Z_A'(0) = \widehat{\zeta}^{c_1,c_2}_{0,q}(iM,0).
\end{equation}
\end{cor}
\begin{proof}
Take the limit of \cref{eq:special} as $s \to 0$.
\end{proof}
We give an example of such a computation in \Cref{sec:example}.

\section{Proof of the Kronecker limit formulas}\label{sec:indefinite}

The method of proof is to compute the Fourier series in $\xi$ for an indefinite theta function with respect to an action by a one-parameter unipotent subgroup $\{T^\xi\}$ of $\SL_2(\R)$, then take a Mellin transform and specialise variables. After taking the Mellin transform, we must allow $\xi$ to be a complex parameter and perform a fairly delicate contour integration. Unlike in the definite case, the Fourier coefficients of the indefinite theta are not elementary functions, which ultimately leads to a more complicated Kronecker limit formula.

We fix the following notation for this section. Let $c_1, c_2 \in \C^2$ satisfying $\ol{c_j}^\top M c_j < 0$, and consider the indefinite theta $\Theta_{p,q}^{c_1,c_2}$ with characteristics $p,q \in \R^2$, as defined in \Cref{def:indefchar}. 
Let $t>0$, $\Omega \in \HH_2^{(1)}$, and $M=\im(\Omega)$. Write the indefinite theta of $t\Omega$ as
\begin{align}
\Theta_{p,q}^{c_1,c_2}(t\Omega) &= \sum_{n \in \Z^2} \rho_{\im(t\Omega)}^{c_1,c_2}\!\left(n+q\right)e\!\left(Q_\Omega(n+q)t+p^\top(n+q)\right) \\
&= \sum_{n \in \Z^2} \rho_{M}^{c_1,c_2}\!\left((n+q)t^{1/2}\right)e\!\left(Q_\Omega(n+q)t+p^\top(n+q)\right),
\end{align}
where
\begin{equation}
\rho_{M}^{c_1,c_2}(v) := \eE\!\left(\frac{c_2^\top M v}{\sqrt{-\foh c_2^\top M c_2}}\right) - \eE\!\left(\frac{c_1^\top M v}{\sqrt{-\foh c_1^\top M c_1}}\right),
\end{equation}
and
\begin{equation}
\eE(z) := \int_0^z e^{-\pi u^2}\,du.
\end{equation}

\subsection{Some lemmas about the Siegel upper half-space}

The statement of our Kronecker limit formula, \Cref{thm:KLF30}, involves a matrix $\Lambda_\Omega^c$ in the Siegel upper half-space $\HH_2^{(0)}$. Its proof will require a few basic lemmas about $\HH_2^{(0)}$.
\begin{lem}\label{lem:toomuch}
Let $\Omega = \smmattwo{\o_{11}}{\o_{12}}{\o_{12}}{\o_{22}} \in \HH_2^{(0)}$. Then 
\begin{equation}\label{eq:toomuch}
\im\!\left(\frac{-1}{\o_{11}}\right)\im\!\left(\frac{\det\Omega}{\o_{11}}\right) > \left(\im\!\left(\frac{\o_{12}}{\o_{11}}\right)\right)^2.
\end{equation}
\end{lem}
\begin{proof}
Express $\Omega$ in terms of its real and imaginary parts,
\begin{equation}
\mattwo{\o_{11}}{\o_{12}}{\o_{12}}{\o_{22}} = \mattwo{n_{11}}{n_{12}}{n_{12}}{n_{22}} + i \mattwo{m_{11}}{m_{12}}{m_{12}}{m_{22}}.
\end{equation}
Note that $m_{11} \neq 0$ because $m_{11} m_{22}-m_{12}^2=\det M > 0$, and thus $\o_{11} \neq 0$. By an algebraic calculation,
\begin{equation}
\im\!\left(\frac{-1}{\o_{11}}\right)\im\!\left(\frac{\det\Omega}{\o_{11}}\right) - \left(\im\!\left(\frac{\o_{12}}{\o_{11}}\right)\right)^2 = \frac{m_{11}m_{22}-m_{12}^2}{n_{11}^2+m_{11}^2}.
\end{equation}
Now, $m_{11}m_{22}-m_{12}^2 = \det M$ is positive, and so is $n_{11}^2+m_{11}^2$. Thus, the inequality \cref{eq:toomuch} holds.
\end{proof}
We will also need the following inequality.
\begin{lem}\label{lem:roots}
Let $\Omega = \smmattwo{\o_{11}}{\o_{12}}{\o_{12}}{\o_{22}} \in \HH_2^{(0)}$. The two roots of $Q_\Omega\smcoltwo{z}{1} = 0$ are $\t_1 = \frac{-\o_{12}+\sqrt{\det(-i\Omega)}}{\o_{11}}$ and $\t_2 = \frac{-\o_{12}-\sqrt{\det(-i\Omega)}}{\o_{11}}$. Then, $\im(\t_1) > 0 > \im(\t_2)$.
\end{lem}
\begin{proof}
We have $Q_\Omega\smcoltwo{z}{1} = \o_{11}z^2 + 2\o_{12}z + \o_{22}$, and the expressions for the roots come from the quadratic formula.

For any complex numbers $\a = a_1+ia_2$ and $\b=b_1+ib_2$, $\left(\im(\a\b)\right)^2 - \im(\a^2)\im(\b^2) = (a_1b_2-a_2b_1)^2 \geq 0$. Thus, $\left(\im(\a\b)\right)^2 \geq \im(\a^2)\im(\b^2)$.

In particular, taking $\a = \frac{1}{\sqrt{-\o_{11}}}$ and $\b = \frac{\sqrt{\det(-i\Omega)}}{\sqrt{-\o_{11}}}$ (for any choice of $\sqrt{-\o_{11}}$), we obtain the inequality 
\begin{align}
\left(\im\!\left(\frac{\sqrt{\det(-i\Omega)}}{\o_{11}}\right)\right)^2 &\geq \im\!\left(\frac{-1}{\o_{11}}\right)\im\!\left(\frac{\det(-i\Omega)}{-\o_{11}}\right) \\ &= \im\!\left(\frac{-1}{\o_{11}}\right)\im\!\left(\frac{\det(\Omega)}{\o_{11}}\right).
\end{align}
Appealing to \Cref{lem:toomuch}, we see by transitivity that
\begin{equation}
\left(\im\!\left(\frac{\sqrt{\det(-i\Omega)}}{\o_{11}}\right)\right)^2 > \left(\im\!\left(\frac{\o_{12}}{\o_{11}}\right)\right)^2.
\end{equation}
By subtracting the left-hand side and factoring, this inequality may be rewritten as $0 > \im(\tau_1)\im(\tau_2)$. So $\im(\tau_1)$ and $\im(\tau_2)$ are always nonzero real numbers with opposite signs. In the special case $\Omega = \smmattwo{i}{0}{0}{i}$, $\tau_1 = i$ and $\tau_2 = -i$. Since $\HH_2^{(0)}$ is connected, we always have $\im(\t_1) > 0 > \im(\t_2)$.
\end{proof}

\subsection{Some integrals involving $\eE(u)$}

We will now prove a few integral formulas that we will need.

\begin{lem}\label{lem:eEeint}
Suppose 
that $\alpha, \beta \in \C$ satisfy $\re\left(\alpha^2 - 2i\beta\right)>0$. 
Then, using the standard branch of the square root function,
\begin{equation}
\int_{0}^\infty \eE(\a t^{1/2})e(\b t)\,dt = \frac{-\alpha}{4\pi i\b\sqrt{\alpha^2-2 i\b}}.\label{eq:mop2}
\end{equation}
\end{lem}
\begin{proof}
By integration by parts,
\begin{align}
\int_{0}^\infty \eE(\a t^{1/2})e(\b t)\,dt 
&= \frac{1}{2\pi i \b} \int_{0}^\infty \eE(\a t^{1/2})\,\frac{d\left(e(\b t)\right)}{dt}\,dt\\
&= \frac{1}{2\pi i \b} \left(\left.\eE(\a t^{1/2})e(\b t)\right|_{t=0}^\infty - \int_0^\infty e^{-\pi \a^2 t}\frac{\a}{2} t^{-1/2}e(\b t)\,dt\right)\\
&= \frac{-\a}{4\pi i\b}\int_0^\infty \exp\left(-\left(\pi\a-2\pi i\b\right)t\right)t^{1/2}\,\frac{dt}{t}\\
&= \frac{-\a}{4\pi i\b}\int_{C} \exp(-u)\left(\frac{u}{\pi \a^2 - 2\pi i\b}\right)^{1/2}\,\frac{du}{u}\\
&= \frac{-\a}{4\pi^{3/2} i\b\sqrt{\a^2-2i\b}}\int_{C} e^{-u}u^{1/2}\,\frac{du}{u},\label{eq:mop1}
\end{align}
where the contour $C$ is a ray from the origin through the point $\a^2-2i\b$. If $z \in \C$ with $x=\re(z)>0$, $s \in \C$ with $\s=\re(s)>0$, and $[z_1,z_2]$ denotes the oriented line segment from $z_1$ to $z_2$, then
\begin{align}
\lim_{N \to \infty} \int_{[0,Nz]} e^{-u}u^{s}\,\frac{du}{u}
&= \lim_{N \to \infty} \left(\int_{[0,Nx]} e^{-u}u^{s}\,\frac{du}{u} + \int_{[Nx,Nz]} e^{-u}u^{s}\,\frac{du}{u}\right) \\
&= \Gamma(s) + \lim_{N \to \infty} \int_{[Nx,Nz]} e^{-u}u^{s}\,\frac{du}{u} \\
&= \Gamma(s) + \lim_{N \to \infty} O\left(e^{-Nx}N^\s\right) \\
&= \Gamma(s).
\end{align}
Thus, in particular, $\int_{C} e^{-u}u^{1/2}\,\frac{du}{u} = \Gamma\left(\foh\right) = \pi^{1/2}$. Plugging this into \cref{eq:mop1} gives \cref{eq:mop2}.
\end{proof}

As usual, let $M = \im(\Omega)$.
Define the following auxiliary function, which will appear as a factor in the integral in the indefinite Kronecker limit formula.
\begin{defn}
For $v \in \C^2$ and $s \in \C$, set
\begin{equation}
\k_{\Omega}^{c}(v,s) := -\int_0^\infty \rho_{M}^{c}\left(vt^{1/2}\right) e\left(Q_\Omega(v)t\right) t^s\,\frac{dt}{t}.
\end{equation}
Also, set
\begin{align}
\k_{\Omega}^{c_1,c_2}(v,s) &:= \k_{\Omega}^{c_2}(v,s) - \k_{\Omega}^{c_1}(v,s) \\
&= \int_0^\infty \rho_{M}^{c_1,c_2}\left(vt^{1/2}\right) e\left(Q_\Omega(v)t\right) t^s\,\frac{dt}{t}.
\end{align}
In the case $s=1$, we will leave out $s$ and set $\k_{\Omega}^{c}(v) := \k_{\Omega}^{c}(v,1)$, $\k_{\Omega}^{c_1,c_2}(v) := \k_{\Omega}^{c_1,c_2}(v,1)$.
\end{defn}

We have the following formula for $\k_{\Lambda}^{c}(v)$ for a particular $\Lambda$ that will come up in the calculation.
\begin{cor}\label{cor:k1}
Let $\Lambda_\Omega^c := \Omega - \frac{i}{Q_M(c)}Mcc^\top M$. Note that $\Lambda_\Omega^c \in\HH_2^{(0)}$ by 
Lemma 3.6 of \cite{kopp1}.
Then,
\begin{equation}
\k_{\Omega}^{c}(v) = \frac{c^\top Mv}{4\pi i \sqrt{-Q_M(c)} Q_\Omega(v)\sqrt{-2iQ_{\Lambda_\Omega^c}(v)}}.
\end{equation}
\end{cor}
\begin{proof}
Follows from \Cref{lem:eEeint}. 
\end{proof}

The following lemma will be needed to evaluate certain integrals.
\begin{lem}
For any real number $\alpha \in \R$, 
\begin{equation}
\int_0^\infty \rho_{M}^{c_1,c_2}\!\left(v \alpha t^{1/2}\right) e\!\left(Q_\Omega(v)\alpha^2t\right) t^s\,\frac{dt}{t} = -\frac{\sgn(\alpha)}{|\alpha|^{2s}}\k_{\Omega}^{c_1,c_2}(v,s).
\end{equation}
\end{lem}
\begin{proof}
Follows from the definition of $\k_{\Omega}^{c_1,c_2}(v,s)$.
\end{proof}

\subsection{Fourier series of a unipotent transform of an indefinite theta function}

Consider the function of $\xi \in \R$ (although $\xi$ will be allowed to be complex later on) and $t \in \R_{\geq 0}$,
\begin{align}
h(\xi,t) &:= \Theta_{\left(T^\xi\right)^\top p, T^{-\xi}q}^{T^{-\xi}c_1,T^{-\xi}c_2}\left(t\left(T^\xi\right)^\top\Omega T^\xi\right) \\
&= \sum_{n \in \Z^2} \rho_{\Omega}^{c_1,c_2}\left(\left(T^\xi n+q\right)t^{1/2}\right) e\left(Q_\Omega(T^\xi n+q)t+p^\top(T^\xi n+q)\right).
\end{align}

Write this function as a Fourier series,
\begin{equation}
h(\xi,t) = \sum_{k=-\infty}^\infty b_k(t)e(k\xi).
\end{equation}
We are ultimately interested in the Mellin transform of this function,
\begin{align}
\widehat{\zeta}_{\left(T^\xi\right)^\top p, T^{-\xi}q}^{T^{-\xi}c_1,T^{-\xi}c_2}\left(\left(T^\xi\right)^\top\Omega T^\xi, s\right) &= \int_0^\infty h(\xi,t)t^s\,\frac{dt}{t} \\
&= \sum_{k=-\infty}^\infty \b_k(s)e(k\xi),
\end{align}
where, as we will show,
\begin{equation}
\b_k(s) := \int_0^\infty b_k(t)t^s\,\frac{dt}{t}.
\end{equation}

Express $\Omega = \smmattwo{\o_{11}}{\o_{12}}{\o_{12}}{\o_{22}}$, $n = \smcoltwo{n_1}{n_2}$, $p = \smcoltwo{p_1}{p_2}$, $q = \smcoltwo{q_1}{q_2}$. Write 
\begin{equation}
h(\xi,t) = \sum_{n_2=-\infty}^\infty h_{n_2}(\xi,t) = h_{0}(\xi,t) + \widetilde{h}(\xi,t), 
\end{equation}
where $h_j(\xi,t)$ is the sum over the terms with $n_2=j$, and $\widetilde{h}(\xi,t)$ is the sum over all the terms where $n_2 \neq 0$. 
Also, assume that $q_1=q_2=0$.

First, calculate $h_0(\xi,t)$:
\begin{align}
h_0(\xi,t) 
&= \sum_{n_1=-\infty}^\infty \rho_\Omega^{c_1,c_2}\!\coltwo{n_1t^{1/2}}{0} e\!\left(\foh \o_{11}n_1^2 t + p_1 n_1 \right).
\end{align}
The $n_1 = 0$ term of this sum vanishes.

We write, for $n_2 \neq 0$,
\begin{align}
&\int_0^1 h_{n_2}(\xi,t)e(-k\xi)\,d\xi \nn \\
&= \int_0^1 \sum_{n_1=-\infty}^\infty \rho_{M}^{c_1,c_2}\!\left(\coltwo{n_1+n_2 \xi}{n_2}t^{1/2}\right) \nn \\ &\hspace{65pt} \cdot e\!\left(Q_\Omega\coltwo{n_1+n_2\xi}{n_2}t+p^\top\!\coltwo{n_1+n_2\xi}{n_2}\right)e(-k\xi)\,d\xi \\
&= \sum_{n_1=0}^{n_2-1} \int_{-\infty}^{\infty} \rho_{M}^{c_1,c_2}\!\left(\coltwo{n_1+n_2 \xi}{n_2}t^{1/2}\right)\nn \\ &\hspace{65pt} \cdot e\!\left(Q_\Omega\coltwo{n_1+n_2\xi}{n_2}t+p^\top\!\coltwo{n_1+n_2\xi}{n_2}\right)e(-k\xi)\,d\xi \\
&= \sum_{n_1=0}^{n_2-1} \int_{-\infty}^{\infty} \rho_{M}^{c_1,c_2}\!\left(\coltwo{n_2 \xi}{n_2}t^{1/2}\right)\nn \\ &\hspace{65pt} \cdot e\!\left(Q_\Omega\coltwo{n_2\xi}{n_2}t+p^\top \coltwo{n_2\xi}{n_2}\right)e\!\left(-k\left(\xi-\frac{n_1}{n_2}\right)\right)\,d\xi \\
&= \left(\sum_{n_1=0}^{n_2-1} e\!\left(\frac{kn_1}{n_2}\right)\right) \int_{-\infty}^{\infty} \rho_{M}^{c_1,c_2}\left(\coltwo{\xi}{1}n_2t^{1/2}\right) \nn \\ &\hspace{125pt} \cdot e\!\left(Q_\Omega\coltwo{\xi}{1}n_2^2t+p^\top \coltwo{\xi}{1}n_2\right)e\!\left(-k\xi\right)\,d\xi.
\end{align}
The exponential sum $\ds \sum_{n_1=0}^{n_2-1} e\!\left(\frac{kn_1}{n_2}\right)$ evaluates to $|n_2|$ if $n_2|k$, and to $0$ otherwise.
Thus, for all $k\in \Z$ (including $k=0$),
\begin{align}\label{eq:goat}
&\int_0^1 \widetilde{h}(\xi,t)e(-k\xi)\,d\xi \\
&= \sum_{n_2 | k} \abs{n_2}\int_{-\infty}^{\infty} \rho_{M}^{c_1,c_2}\!\left(\coltwo{\xi}{1}n_2t^{1/2}\right) e\!\left(Q_\Omega\!\coltwo{\xi}{1}n_2^2t+p^\top\!\coltwo{\xi}{1}n_2\right)e\!\left(-k\xi\right)\,d\xi. \nn 
\end{align}
Our convention here is that a sum over $n_2|k$ ranges over \textit{both positive and negative} $n_2$, and over all integers when $k=0$.

\subsection{Shifting the contour vertically}

Fix a positive real number $\l$ to be specified later. Let $C^+$ ($C^-$) be the contour consisting of the horizontal line $\im(z) = \l$ ($\im(z) = -\l$), oriented towards the right half-plane.
For each $d_1,d_2 \in \Z$, $d_2 \neq 0$, let $C(d_1,d_2)$ be $C^+$ if $d_1d_2>0$ or $d_1=0$ and $d_2>0$; let $C(d_1,d_2)$ be $C^-$ if $d_1d_2<0$ or $d_1=0$ and $d_2<0$.
The integrands in \cref{eq:goat} approach zero as $\re(\xi) \to \pm \infty$, so we may rewrite this formula using contour integrals
\begin{align}
&\int_0^1 \widetilde{h}(\xi,t)e(-k\xi)\,d\xi \\
&= \sum_{n_2 | k} \abs{n_2}\int_{C\left(\frac{k}{n_2},n_2\right)} \rho_{M}^{c_1,c_2}\!\left(\coltwo{\xi}{1}n_2t^{\foh}\right) e\!\left(Q_\Omega\coltwo{\xi}{1}n_2^2t+p^\top \coltwo{\xi}{1}n_2\right)e\!\left(-k\xi\right)\,d\xi.\nn
\end{align}

\subsection{Taking Mellin transforms term-by-term}

To calculate the Mellin transform of $h_0(\xi,t)$, we need to check absolute convergence to justify reversing the order of summation/integration. 
\begin{prop} 
If $\s = \re(s) > \foh$, then
\begin{equation}
\int_0^\infty \sum_{n_1=-\infty}^{\infty} \abs{\rho_\Omega^{c_1,c_2}\!\coltwo{n_1t^{1/2}}{0} e\!\left(\foh \o_{11}n_1^2 t + p_1 n_1 \right)}t^\s\,\frac{dt}{t} < \infty.
\end{equation}
\end{prop}
\begin{proof}
We bound the integral as follows.
\begin{align}
&\int_0^\infty \sum_{n_1=-\infty}^{\infty} \abs{\rho_\Omega^{c_1,c_2}\!\coltwo{n_1t^{1/2}}{0} e\!\left(\foh \o_{11}n_1^2 t + p_1 n_1 \right)}t^\s\,\frac{dt}{t} \\
&= \int_0^\infty \sum_{n_1=-\infty}^{\infty} \abs{\rho_\Omega^{c_1,c_2}\!\coltwo{t^{1/2}}{0} e\!\left(\foh \o_{11} t \right)}\left(\frac{t}{n_1^2}\right)^\s\,\frac{dt}{t} \\
&= \left(\sum_{n_1=-\infty}^{\infty} \abs{n_1}^{-2\s}\right)\left(\int_0^\infty \abs{\rho_\Omega^{c_1,c_2}\coltwo{t^{1/2}}{0} e\!\left(\foh \o_{11} t \right)}t^\s\,\frac{dt}{t}\right) \\
&< \infty.
\end{align}
The sum converges for $\s > \foh$, and the integral converges for $\s > 0$ (because the integrand approaches a constant at $t \to 0$ and decays exponentially as $t \to \infty$).
\end{proof}

Therefore, we can switch the sum and the integral. Using \Cref{lem:eEeint} and dropping the subscript on $n_1$,
\begin{align}
\int_0^\infty h_0(\xi,t)t^s\,\frac{dt}{t} &= -\sum_{n \in \Z \setminus \{0\}} \frac{\sgn(n)e(p_1 n)}{\abs{n}^{2s}}\k_{\Omega}^{c_1,c_2}\!\left(\coltwo{1}{0},s\right) \\
&= -\left(\Li_{2s}(e(p_1))-\Li_{2s}(e(-p_1))\right)\k_{\Omega}^{c_1,c_2}\!\left(\coltwo{1}{0},s\right).
\end{align}

Next, we're going to calculate the Mellin transform of $\widetilde{h}(\xi,t)$. We need an absolute convergence result to justify our calculation here, too.

\begin{prop}\label{prop:bigabs}
Suppose $\s = \re(s) > \foh$. Then,
\begin{align}
\sum_{k \in \Z} \sum_{\substack{n_2 | k \\ n_2 \neq 0}} \int_0^\infty \int_{C\left(\frac{k}{n_2},n_2\right)}&\left|\rho_{M}^{c_1,c_2}\!\left(\coltwo{\xi}{1}n_2t^{1/2}\right)\right.\nn \\& \left. \ \ \cdot e\!\left(Q_\Omega\coltwo{\xi}{1}n_2^2t+p^\top \coltwo{\xi}{1}n_2\right)e\!\left(-k\xi\right)t^s\right|\,\frac{dt}{t}d\xi < \infty.
\end{align}
\end{prop}
\begin{proof}
Let
\begin{align}
K^{\pm} &:= \int_0^\infty \int_{C^\pm}\abs{\rho_{M}^{c_1,c_2}\!\left(\coltwo{\xi}{1}t^{1/2}\right) e\!\left(Q_\Omega\!\coltwo{\xi}{1}t\right)}t^\s\,d\xi\frac{dt}{t} \\
&< \infty. 
\end{align}
Set $K := \max\{K^+, K^-\}$. We have
\begin{align}
&\sum_{k \in \Z} \sum_{\substack{n_2 | k \\ n_2 \neq 0}} \int_0^\infty \int_{C\left(\frac{k}{n_2},n_2\right)}\left|\rho_{M}^{c_1,c_2}\!\left(\coltwo{\xi}{1}n_2t^{1/2}\right) \right. \nn \\ &\hspace{130pt}\left. \cdot e\!\left(Q_\Omega\coltwo{\xi}{1}n_2^2t+p^\top \!\coltwo{\xi}{1}n_2\right)e\!\left(-k\xi\right)t^s\right|\,\frac{dt}{t}d\xi \nn \\
&= \sum_{k \in \Z} \sum_{\substack{n_2 | k \\ n_2 \neq 0}} \int_0^\infty \int_{C\left(\frac{k}{n_2},n_2\right)}\left|\rho_{M}^{c_1,c_2}\!\left(\coltwo{\xi}{1}n_2t^{1/2}\right) e\!\left(Q_\Omega\coltwo{\xi}{1}n_2^2t\right)\right| \nn \\ &\hspace{130pt} \cdot e^{-2\pi\l k}t^\s\,\frac{dt}{t}d\xi \\
&= \sum_{k \in \Z} \sum_{\substack{n_2 | k \\ n_2 \neq 0}} \int_0^\infty \int_{C\left(\frac{k}{n_2},n_2\right)}\abs{\rho_{M}^{c_1,c_2}\!\left(\coltwo{\xi}{1}t^{1/2}\right) e\!\left(Q_\Omega\coltwo{\xi}{1}t\right)} \nn \\ &\hspace{130pt} \cdot e^{-2\pi\l k}\left(\frac{t}{n_2^2}\right)^\s\,\frac{dt}{t}d\xi \\
&\leq K \sum_{k \in \Z} \sum_{\substack{n_2 | k \\ n_2 \neq 0}} e^{-2\pi\l k}n_2^{-2\s} \\
&= K \sum_{d_1 \in \Z} \sum_{d_2 \in \Z \setminus \{0\}} e^{-2\pi\l \abs{d_1d_2}}d_2^{-2\s} \\
&< \infty.
\end{align}
The proposition is proved.
\end{proof}

Now we may justify taking the Mellin transform of the Fourier series term-by-term.
It follows from \Cref{prop:bigabs} that 
\begin{align}
\widehat{\zeta}_{\left(T^\xi\right)^\top p, 0}^{T^{-\xi}c_1,T^{-\xi}c_2}\left(\left(T^\xi\right)^\top\Omega T^\xi, s\right) &= \int_0^\infty h(\xi,t)t^s\,\frac{dt}{t} \\
&= \sum_{k=-\infty}^\infty \b_k(s)e(k\xi),\label{eq:fish}
\end{align}
where $\ds \b_k(s) := \int_0^\infty b_k(t)t^s\,\frac{dt}{t}$.
Define $\ds \widetilde{\b}_k(s) := \int_0^\infty \widetilde{b}_k(t)t^s\,\frac{dt}{t}$; then,
\begin{equation}
\b_k(s) = \left\{\begin{array}{ll}
-\left(\Li_{2s}(e(p_1))-\Li_{2s}(e(-p_1))\right)\k_{\Omega}^{c_1,c_2}\!\left(\coltwo{1}{0},s\right)+\widetilde{\b}_0(s) & \mbox{ if } k = 0, \\
\widetilde{\b}_k(s) & \mbox{ if } k \neq 0.
\end{array}\right.
\end{equation}
\Cref{prop:bigabs} also implies that we can switch the order of integration to compute 
\begin{align}
\widetilde{\b}_k(s) &= \int_0^\infty\int_0^1 \widetilde{h}(\xi,t)e(-k\xi)\,d\xi \,t^s \frac{dt}{t} \\
&= \sum_{n_2 | k} \abs{n_2} \int_{C\left(\frac{k}{n_2},n_2\right)} e\!\left(n_2 p^\top\coltwo{\xi}{1}-k\xi\right)\left(-\sgn(n_2)\abs{n_2}^{-2s}\k_{\Omega}^{c_1,c_2}(\xi,s)\right) \,d\xi \\
&= -\sum_{n_2 | k} \frac{\sgn(n_2)}{\abs{n_2}^{2s-1}} \int_{C\left(\frac{k}{n_2},n_2\right)} e\!\left(n_2 (p_1\xi+p_2)-k\xi\right)\k_{\Omega}^{c_1,c_2}(\xi,s) \,d\xi.
\end{align}

\subsection{Series manipulations}

In this subsection, we set $\xi=0$ in \cref{eq:fish}. We will manipulate the right-hand side of this equation to prove \Cref{thm:KLF30}. First of all, we have
\begin{align}
\widehat{\zeta}_{p,0}^{c_1,c_2}\left(\Omega, s\right)
&= \sum_{k=-\infty}^\infty \b_k(s) \\
&= -\left(\Li_{2s}(e(p_1))-\Li_{2s}(e(-p_1))\right)\k_{\Omega}^{c_1,c_2}\!\left(\coltwo{1}{0},s\right) + \sum_{k=-\infty}^\infty \widetilde{\b}_k(s).
\end{align}
We will rewrite the sum of the $\widetilde{\b}_k(s)$ using the substitution $(d_1,d_2) = (\frac{k}{n_2},n_2)$. The following manipulation is legal by \Cref{prop:bigabs}.
\begin{align}
\sum_{k=-\infty}^\infty \widetilde{\b}_k(s)
&= -\sum_{k \in \Z} \sum_{\substack{n_2 | k \\ n_2 \neq 0}} \frac{\sgn(n_2)}{\abs{n_2}^{2s-1}} \int_{C\left(\frac{k}{n_2},n_2\right)} e\!\left(n_2 (p_1\xi+p_2)-k\xi\right)\k_{\Omega}^{c_1,c_2}(\xi,s) \,d\xi \\
&= -\sum_{d_1 \in \Z} \sum_{d_2 \in \Z \setminus \{0\}} \frac{\sgn(d_2)}{\abs{d_2}^{2s-1}} \int_{C\left(d_1,d_2\right)} e\!\left(d_2 (p_1\xi+p_2)-d_1d_2\xi\right)\k_{\Omega}^{c_1,c_2}(\xi,s) \,d\xi.
\end{align}
Split up the series into four pieces.
\begin{align}
\sum_{k=-\infty}^\infty \widetilde{\b}_k(s)
&= -\sum_{d_1 > 0} \sum_{d_2 > 0} \frac{e(d_2p_2)}{\abs{d_2}^{2s-1}} \int_{C^-} e\!\left(-(d_1-p_1)d_2\xi\right)\k_{\Omega}^{c_1,c_2}(\xi,s) \,d\xi \nn\\
&\ \ \ +\sum_{d_1 > 0} \sum_{d_2 < 0} \frac{e(d_2p_2)}{\abs{d_2}^{2s-1}} \int_{C^+} e\!\left(-(d_1-p_1)d_2\xi\right)\k_{\Omega}^{c_1,c_2}(\xi,s) \,d\xi \nn\\
&\ \ \ -\sum_{d_1 \leq 0} \sum_{d_2 > 0} \frac{e(d_2p_2)}{\abs{d_2}^{2s-1}} \int_{C^+} e\!\left(-(d_1-p_1)d_2\xi\right)\k_{\Omega}^{c_1,c_2}(\xi,s) \,d\xi \nn\\
&\ \ \ +\sum_{d_1 \leq 0} \sum_{d_2 < 0} \frac{e(d_2p_2)}{\abs{d_2}^{2s-1}} \int_{C^-} e\!\left(-(d_1-p_1)d_2\xi\right)\k_{\Omega}^{c_1,c_2}(\xi,s) \,d\xi \\
&= -\sum_{d_1 > 0} \sum_{d_2 > 0} \frac{e(d_2p_2)}{\abs{d_2}^{2s-1}} \int_{C^+} e\!\left((d_1-p_1)d_2\xi\right)\k_{\Omega}^{c_1,c_2}(-\xi,s) \,d\xi \nn\\
&\ \ \ +\sum_{d_1 > 0} \sum_{d_2 < 0} \frac{e(d_2p_2)}{\abs{d_2}^{2s-1}} \int_{C^+} e\!\left(-(d_1-p_1)d_2\xi\right)\k_{\Omega}^{c_1,c_2}(\xi,s) \,d\xi \nn\\
&\ \ \ -\sum_{d_1 \leq 0} \sum_{d_2 > 0} \frac{e(d_2p_2)}{\abs{d_2}^{2s-1}} \int_{C^+} e\!\left(-(d_1-p_1)d_2\xi\right)\k_{\Omega}^{c_1,c_2}(\xi,s) \,d\xi \nn\\
&\ \ \ +\sum_{d_1 \leq 0} \sum_{d_2 < 0} \frac{e(d_2p_2)}{\abs{d_2}^{2s-1}} \int_{C^+} e\!\left((d_1-p_1)d_2\xi\right)\k_{\Omega}^{c_1,c_2}(-\xi,s) \,d\xi \\
&= -\sum_{d_1 > 0} \sum_{d_2 > 0} \frac{e(d_2p_2)}{d_2^{2s-1}} \int_{C^+} e\!\left((d_1-p_1)d_2\xi\right)\k_{\Omega}^{c_1,c_2}(-\xi,s) \,d\xi \nn\\
&\ \ \ +\sum_{d_1 > 0} \sum_{d_2 > 0} \frac{e(-d_2p_2)}{d_2^{2s-1}} \int_{C^+} e\!\left((d_1-p_1)d_2\xi\right)\k_{\Omega}^{c_1,c_2}(\xi,s) \,d\xi \nn\\
&\ \ \ -\sum_{d_1 \geq 0} \sum_{d_2 > 0} \frac{e(d_2p_2)}{d_2^{2s-1}} \int_{C^+} e\!\left((d_1+p_1)d_2\xi\right)\k_{\Omega}^{c_1,c_2}(\xi,s) \,d\xi \nn\\
&\ \ \ +\sum_{d_1 \geq 0} \sum_{d_2 > 0} \frac{e(-d_2p_2)}{d_2^{2s-1}} \int_{C^+} e\!\left((d_1+p_1)d_2\xi\right)\k_{\Omega}^{c_1,c_2}(-\xi,s) \,d\xi.
\end{align}
Now, move the contour integral outside the sums, and rewrite the series as
\begin{align}
&\sum_{k=-\infty}^\infty \widetilde{\b}_k(s)
= \int_{C+} \left( \sum_{d_2 \geq 0} \frac{e(-p_2+p_1\xi)^{d_2}}{d_2^{2s-1}}\k_{\Omega}^{c_1,c_2}(-\xi,s) - \sum_{d_2 \geq 0} \frac{e(p_2+p_1\xi)^{d_2}}{d_2^{2s-1}}\k_{\Omega}^{c_1,c_2}(\xi,s)\right. \\
&\hspace{20pt}+\sum_{d_1 > 0} \sum_{d_2 > 0} \frac{1}{d_2^{2s-1}}\left(\left(-e\!\left((d_1-p_1)\xi+p_2\right)^{d_2}+e\!\left((d_1+p_1)\xi-p_2\right)^{d_2}\right)\k_{\Omega}^{c_1,c_2}(-\xi,s)\right. \nn \\
&\hspace{85pt}+\left.\left.\left(e\!\left((d_1-p_1)\xi-p_2\right)^{d_2}-e\!\left((d_1+p_1)\xi+p_2\right)^{d_2}\right)\k_{\Omega}^{c_1,c_2}(\xi,s)\right)\right)\,d\xi. \nn
\end{align}
Setting $s=1$, we obtain
\begin{align}
&\sum_{k=-\infty}^\infty \widetilde{\b}_k(1) \\
&= \int_{C+} \left(-\log(1-e(-p_2+p_1\xi))\k_{\Omega}^{c_1,c_2}\!\coltwo{-\xi}{1}+\log(1-e(p_2+p_1\xi))\k_{\Omega}^{c_1,c_2}\!\coltwo{\xi}{1}\right. \nn \\
&\hspace{10pt}+\sum_{d_1 = 1}^\infty \left(\left(\log\left(1-e\!\left((d_1-p_1)\xi+p_2\right)\right)-\log\left(1-e\left((d_1+p_1)\xi-p_2\right)\right)\right)\k_{\Omega}^{c_1,c_2}\!\coltwo{-\xi}{1}\right. \nn \\
&\hspace{15pt}\left.\left.\left(-\log\left(1-e\!\left((d_1-p_1)\xi-p_2\right)\right)+\log\left(1-e\!\left((d_1+p_1)\xi+p_2\right)\right)\right)\k_{\Omega}^{c_1,c_2}\!\coltwo{\xi}{1}\right)\right)\,d\xi. \nn
\end{align}
We want to write this sum of logarithms as a logarithm of a product, but there is the issue of the choice of branch. In order to make a clear choice, let
\begin{equation}
\varphi_{p_1,p_2}(\xi) := \left(1-e\left(p_1\xi+p_2\right)\right)\prod_{d=1}^\infty\frac{1-e\left((d+p_1)\xi+p_2\right)}{1-e\left((d-p_1)\xi-p_2\right)}
\end{equation}
for $\xi \in \HH$. 
This is a function on the upper half-plane which is never zero, and the upper half-plane is simply connected, so it has a choice of continuous logarithm. Let $\left(\Log \varphi_{p_1,p_2}\right)(\xi)$ be the branch such that
\begin{equation}
\lim_{\xi \to i\infty} \left(\Log\varphi_{p_1,p_2}\right)(\xi) = \left\{\begin{array}{ll}
\log(1-e(p_2)) & \mbox{ if } p_1 = 0, \\
0 & \mbox{ if } p_1 \neq 0.
\end{array}\right.
\end{equation}
Here $\log(1-e(p_2))$ is the standard principal branch. Thus,
\begin{align}
\sum_{k=-\infty}^\infty \widetilde{\b}_k(1) &= 
\int_{C+} \left(-\left(\Log\varphi_{p_1,-p_2}\right)(\xi)\cdot\k_{\Omega}^{c_1,c_2}\!\coltwo{-\xi}{1}\right.\nn \\
&\hspace{50pt}\left.+\left(\Log\varphi_{p_1,p_2}\right)(\xi)\cdot\k_{\Omega}^{c_1,c_2}\!\coltwo{\xi}{1}\right)\,d\xi.
\end{align}
Adding back the other piece of $\beta_0(1)$ into $\ds \widehat{\zeta}^{c_1,c_2}_{p,0}(\Omega,1)=\sum_{k=-\infty}^\infty \b_k(1)$, we obtain
\begingroup
\allowdisplaybreaks[0]
\begin{align}\label{eq:slow}
\widehat{\zeta}^{c_1,c_2}_{p,0}(\Omega,1) &= 
-\left(\Li_2(e(p_1))-\Li_2(e(-p_1))\right)\k^{c_1,c_2}_{\Omega}\!\coltwo{1}{0} \\
&\ \ \ +\int_{C+} \left(-\left(\Log\varphi_{p_1,-p_2}\right)(\xi)\cdot\k_{\Omega}^{c_1,c_2}\!\coltwo{-\xi}{1}\right.\\
&\hspace{56pt}\left.+\left(\Log\varphi_{p_1,p_2}\right)(\xi)\cdot\k_{\Omega}^{c_1,c_2}\!\coltwo{\xi}{1}\right)\,d\xi.
\end{align}
\endgroup

\subsection{Collapsing the contour onto the branch cuts}

We could declare ourselves done at this point. \Cref{eq:slow} is a formula for $\widehat{\zeta}^{c_1,c_2}_{p,0}(\Omega,1)$, as we desired, and it appears very difficult to evaluate or simplify the contour integral in any way. However, \cref{eq:slow} is not a useful formula for computation because the integral converges slowly. The integrand decays polynomially as $\xi \to \pm \infty$ along the horizontal contour $C^+$.

We will obtain a Kronecker limit formula with rapid convergence by shifting the contour so that the integrand decays exponentially. In doing so, we will also split up the formula as a difference of a $c_1$-piece and a $c_2$-piece. The movement of the contour is shown in \Cref{fig:contour}.

Let $\Lambda_\Omega^c := \Omega - \frac{i}{Q_M(c)}Mcc^\top M$ for $c=c_1,c_2$, as we did in \Cref{cor:k1}. 
Factor the quadratic polynomial $Q_{\Lambda_\Omega^c}\smcoltwo{\xi}{1}$ in $\xi$,
\begin{equation}
Q_{\Lambda_\Omega^c}\!\coltwo{\xi}{1}= \alpha(c)(\xi-\t_1(c))(\xi-\t_2(c)).
\end{equation}
Since $\Lambda_\Omega^c \in\HH_2^{(0)}$ by Lemma 3.6 of \cite{kopp1}, 
we know by \Cref{lem:roots} that we may choose
$\t_1(c)$ to be in the upper half-plane and $\t_2(c)$ in the lower half-plane.

The complex function $\xi \mapsto \k^{c}_{\Omega}\!\smcoltwo{\xi}{1}$ has branch cuts along the vertical ray from $\tau_1(c)$ to $i\infty$ and the vertical ray from $\tau_2(c)$ to $-i\infty$.
We check that this function is holomorphic away from these branch cuts. Since 
$\k_{\Omega}^{c}\!\smcoltwo{\xi}{1}$ has simple poles 
at the roots $\xi = r_1,r_2$ of $Q_\Omega\!\smcoltwo{\xi}{1} = 0$, we must check that the residues at the poles cancel when taking the difference
$
\k_{\Omega}^{c_1,c_2}\!\smcoltwo{\xi}{1} = \k_{\Omega}^{c_2}\!\smcoltwo{\xi}{1} - \k_{\Omega}^{c_1}\!\smcoltwo{\xi}{1}
$.
We have
\begin{align}
&\underset{\xi \to r_1}{\res}\k_{\Omega}^{c}\!\coltwo{\xi}{1} \nn \\
&= \lim_{\xi \to r_1} (\xi-r_1)\frac{c^\top M \coltwo{\xi}{1}}{2\pi i Q_\Omega\!\coltwo{\xi}{1}\sqrt{\left(c^\top M \coltwo{\xi}{1}\right)^2-2iQ_M(c)Q_\Omega\!\colxi}} \\
&= \lim_{\xi \to r_1} \frac{c^\top M \coltwo{\xi}{1}}{\pi i \omega_{11}(\xi-r_2)\sqrt{\left(c^\top M \coltwo{\xi}{1}\right)^2-2iQ_M(c)Q_\Omega\!\colxi}} \\
&= \frac{1}{\pi i \omega_{11}(r_1-r_2)},
\end{align}
and similarly, $\underset{\xi \to r_2}{\res}\k_{\Omega}^{c}\!\smcoltwo{\xi}{1} = \frac{1}{\pi i \omega_{11}(r_2-r_1)}$. These residues do not depend on $c$, so they cancel, and $\k_{\Omega}^{c_1,c_2}\!\smcoltwo{\xi}{1}$ is holomorphic at $r_1$ and $r_2$.

\begin{figure}\label{fig:contour}
\includegraphics[scale=.57]{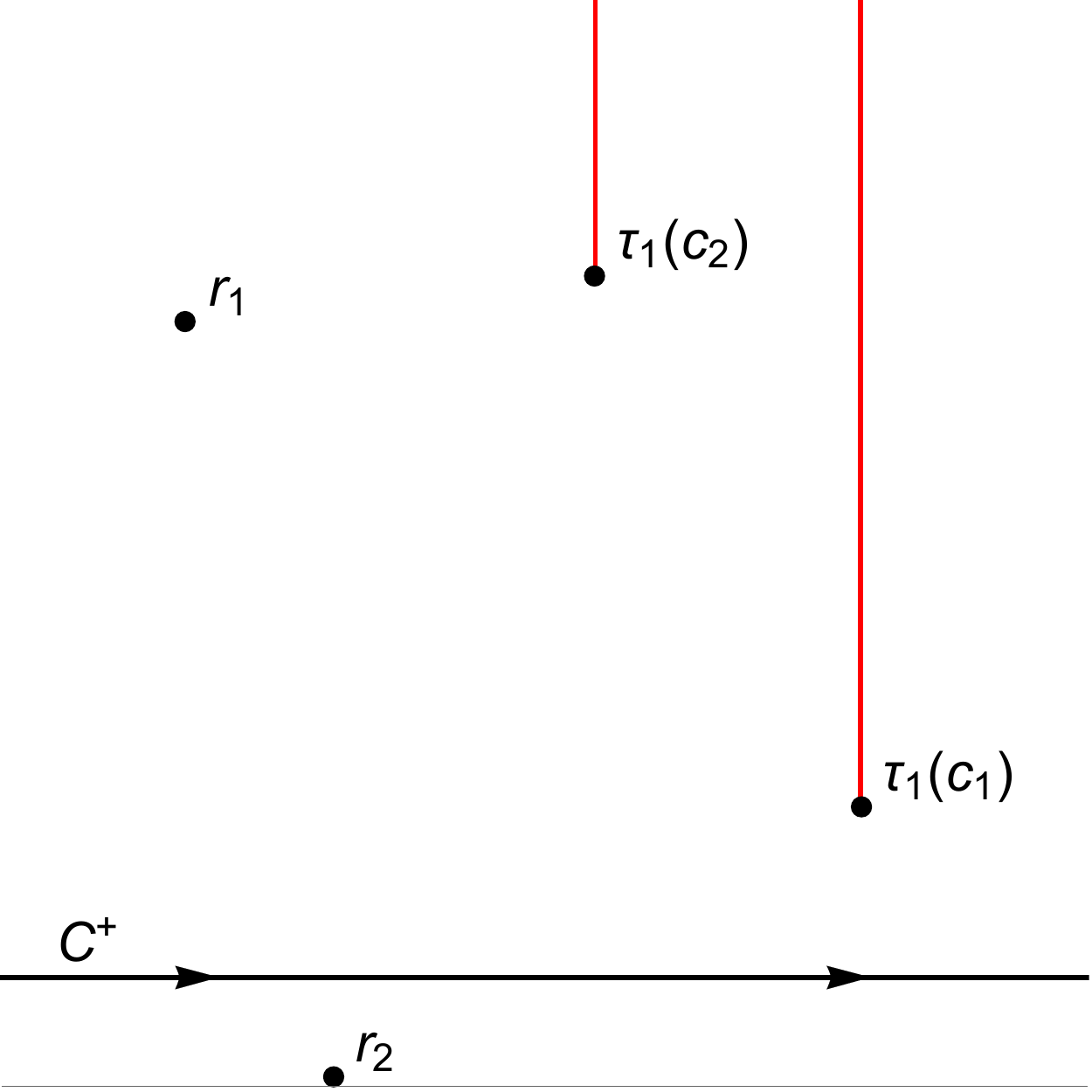} \ \ \
\includegraphics[scale=.57]{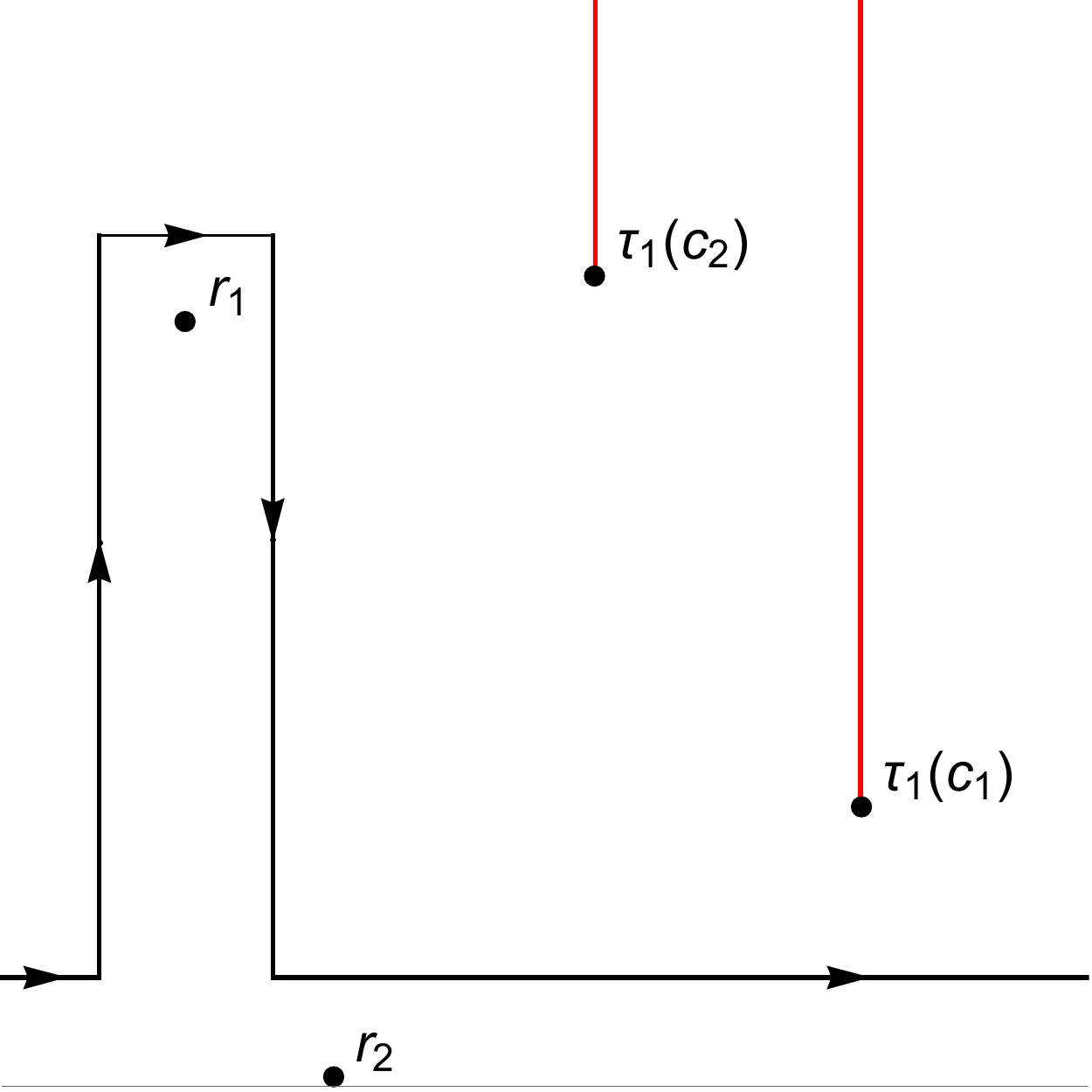} \\ \ \\
\includegraphics[scale=.57]{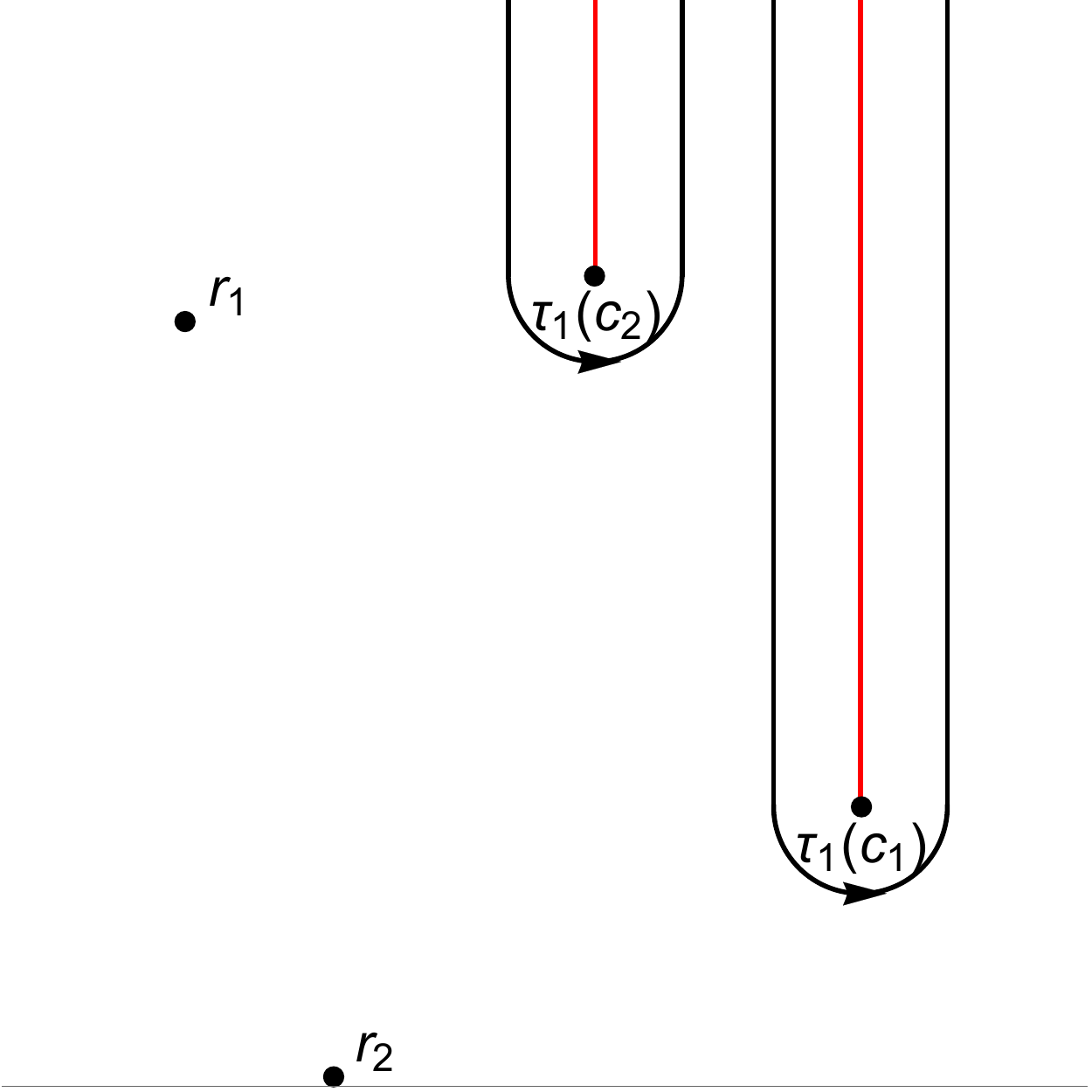} \ \ \
\includegraphics[scale=.57]{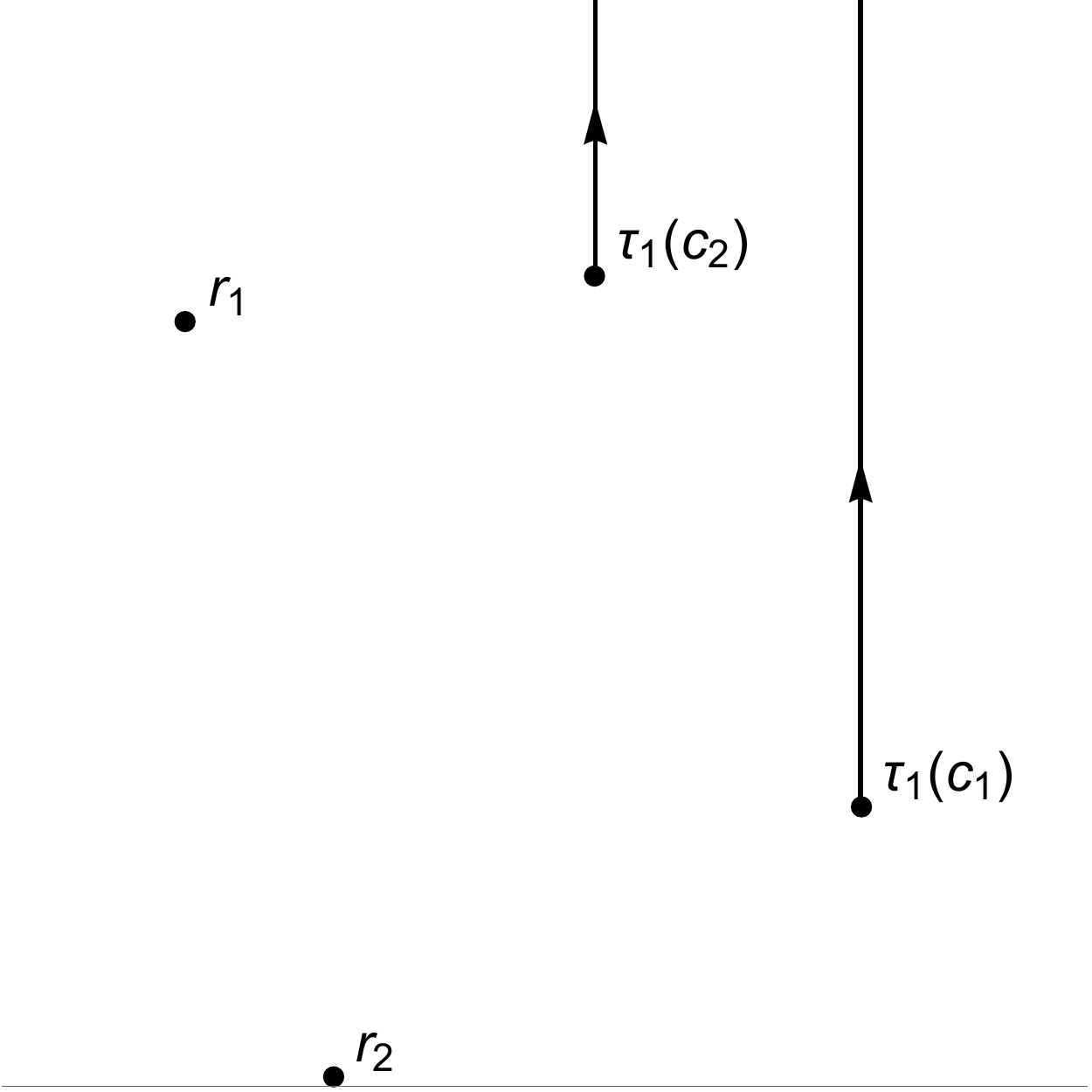}
\caption[Shifting the countour of integration.]{The contour $C^+$ is moved above the poles of $\kappa^c_\Omega\smcoltwo{\xi}{1}$, then collapsed onto branch cuts.}
\end{figure}

Move the countours of integration above the zeros of $Q_\Omega\smcoltwo{\pm \xi}{1}$. Now we may safely split up the integral into a term for $c_1$ and a term for $c_2$.

Now we retract the integral onto the branch cut. 
As $\xi = \pm\t^\pm + \e$ and $\e \to 0$, the denominator of the integrand blows up like $\e^{1/2}$, so the integral converges. The integrand changes sign when we cross the branch cut. Thus, \cref{eq:slow} becomes
\begin{align}
\widehat{\zeta}^{c_1,c_2}_{p,0}(\Omega,1) &= 
I^+(c_2) - I^-(c_2) - I^+(c_1) + I^-(c_1),
\end{align}
where
\begin{align}
I^{\pm}(c) &:= 
-\Li_2(e(\pm p_1))\k^{c}_{\Omega}\coltwo{1}{0} \nn \\
&\ \ \ +2i\int_0^\infty \left(\Log \varphi_{p_1,\pm p_2}\right)(\pm\tau^{\pm}(c)+it) \k_{\Omega}^{c}\coltwo{\pm\left(\tau^{\pm}(c)+it\right)}{1}\,dt.
\end{align}
We have now proven \Cref{thm:KLF30}. \Cref{thm:KLF3} follows by specializing the variables, setting $\Omega = iM$ and restricting to $c_1,c_2 \in \R^g$. \Cref{thm:KLF40} and \Cref{thm:KLF4} both follow by application of the functional equation (\Cref{thm:zetafun}).

\section{Example}\label{sec:example}
We conclude with an example to show how to use the Kronecker limit formula for indefinite zeta functions to compute Stark units. This example was introduced in Section 7.1 of \cite{kopp1}.

Let $K = \Q(\sqrt{3})$, so $\OO_K = \Z[\sqrt{3}]$, and let $\cc = 5\OO_K$.
The ray class group $\Cl_{\cc\infty_2}(\OO_K) \isom \Z/8\Z$.
The fundamental unit $\e = 2+\sqrt{3}$ is totally positive: $\e\e' = 1$. It has order $3$ modulo $5$: $\e^3 = 26+15\sqrt{3} \con 1 \Mod{5}$.
In this section, we use the Kronecker limit formula for indefinite zeta functions to compute $Z_I'(0)$, where $I$ is the principal ray class of $\Cl_{\cc\infty_2}(\OO_K)$.

Let $M = \smmattwo{2}{0}{0}{-6}$, $q = \smcoltwo{1/5}{0}$, and $c_1\in \R^2$ any column vector with the property that $c_1^\top M c_1 < 0$, such as $c_1 = \smcoltwo{0}{1}$.
By \Cref{cor:special} and the discussion in Section 7.1 of \cite{kopp1}, we have
\begin{equation}\label{eq:zetaeg}
Z_I'(0) = \widehat{\zeta}^{c_1,P^3c_1}_{0,q}(\Omega,0),
\end{equation}
where $\Omega = iM$ and $P = \smmattwo{2}{3}{1}{2}$.

Now we want to use \Cref{thm:KLF4} to compute the right-hand side of \cref{eq:zetaeg}.  If we try to do so directly, we obtain $P^3 = \smmattwo{26}{45}{15}{26}$, $P^3c_1 = \smcoltwo{45}{26}$, $\ol{\Omega}P^3c_1 = 6i \smcoltwo{-15}{26}$, and $\Lambda_{-\Omega^{-1}}^{\ol{\Omega}P^3c_1} = -i\smmattwo{675}{390}{390}{676/3}$. The root of $Q_{\Lambda_{-\Omega^{-1}}^{\ol{\Omega}P^3 c_1}}\!\smcoltwo{\xi}{1}$ in the upper half-plane (equivalently, the branch point of $\kappa^{\ol{\Omega}P^3 c_1}_{-\Omega^{-1}}\!\smcoltwo{\xi}{1}$ in the upper half-plane) is $\xi = \frac{-2340+i\sqrt{3}}{4053}$, which is very close to the real axis. That means we'd need to use about $\frac{\log(10)N}{2\pi\sqrt{3}/4053} \approx 857.5N$ terms in the product expansion of $\varphi_{-q_1,q_2}(\xi)$ to compute $Z_I'(0)$ to $N$ decimal places of accuracy. We technically have exponential decay, but it's not very useful.

It is much more practical to break up the zeta function into pieces. We can also improve the rate of convergence by choosing $c_1$ optimally; here, we will use $c = \smcoltwo{-1}{1}$ in place of $c_1$. We have
\begin{align}
Z_I'(0) 
&= \widehat{\zeta}^{c,P^3c}_{0,q}(\Omega,0) \\
&= \widehat{\zeta}^{c,Pc}_{0,q}(\Omega,0) + \widehat{\zeta}^{Pc,P^2c}_{0,q}(\Omega,0) + \widehat{\zeta}^{P^2c,P^3c}_{0,q}(\Omega,0) \\
&= \widehat{\zeta}^{c,Pc}_{0,q}(\Omega,0) + \widehat{\zeta}^{c,Pc}_{0,q'}(\Omega,0) + \widehat{\zeta}^{c,Pc}_{0,q''}(\Omega,0),
\end{align}
where $q = \frac{1}{5}\smcoltwo{1}{0}$, $q'= \frac{1}{5}\smcoltwo{2}{1}$, and $q''= \frac{1}{5}\smcoltwo{2}{4}$ are obtained from the residues of the global units $\e^0, \e^1, \e^2$ modulo $5$. 

Now, we have $\kappa_{-\Omega^{-1}}^{\ol{\Omega} c}\!\smcoltwo{\xi}{1} = \frac{-3\sqrt{6}(\xi-1)}{\pi(3\xi^2-1)\sqrt{3\xi^2-3\xi+1}}$ and $\kappa_{-\Omega^{-1}}^{\ol{\Omega}Pc}\!\smcoltwo{\xi}{1} = \frac{3\sqrt{6}(\xi+1)}{\pi(3\xi^2-1)\sqrt{3\xi^2+3\xi+1}}$, 
with branch points in the upper half-plane at $\xi = \frac{3+i\sqrt{3}}{6}$ and $\xi = \frac{-3+i\sqrt{3}}{6}$, respectively. 
We thus need to use about $\frac{\log(10)N}{2\pi\sqrt{3}/6} \approx 1.269N$ terms in the product expansion of each of the functions $\varphi_{-q_1,-q_2}(\xi)$, $\varphi_{-q_1',-q_2'}(\xi)$, and $\varphi_{-q_1'',-q_2''}(\xi)$ to compute $Z_I'(0)$ to $N$ decimal places of accuracy by this method. 
For $q, q', q''$, we computed the corresponding values of the integrals $J(c), J'(c), J''(c)$ and $J(Pc), J'(Pc), J''(Pc)$ given by \cref{eq:klf4int}. The computation was performed in Mathematica using numerical integral of the first 40 terms of the product expansion of each $\varphi$. For the differences of the two integrals, we obtain
\begin{align}
J(Pc)-J(c) \approx \ &-0.05923843917544488329354507987 \nn \\ &+ 3.65687839020311786132893850239 i, \\
J'(Pc)-J'(c) \approx \ &-1.33733021085943469210685014899 \nn \\ &+ 0.52477812529424663387556899167 i, \mbox{ and}\\
J''(Pc)-J''(c) \approx \ &2.64057587271922212456484190607 \nn \\ &+ 0.52477812529424663387556899167 i.
\end{align}
For the ray class zeta value, we thus calculate using \Cref{thm:KLF4} that
\begin{align}
Z_I'(0) &= \widehat{\zeta}^{c,Pc}_{0,q}(\Omega,0) + \widehat{\zeta}^{c,Pc}_{0,q'}(\Omega,0) + \widehat{\zeta}^{c,Pc}_{0,q''}(\Omega,0) \\
&= \frac{2i}{\sqrt{\det{M}}}\im(J(Pc)-J(c))+\frac{2i}{\sqrt{\det{M}}}(J'(Pc)-J'(c))+\frac{2i}{\sqrt{\det{M}}}(J''(Pc)-J''(c)) \\
&= \frac{1}{2\sqrt{3}}\im\left(J(Pc)-J(c)+J'(Pc)-J'(c)+J''(Pc)-J''(c)\right) \\
&\approx 1.35863065339220816259511308230.
\end{align}

This agrees (to 30 decimal digits) with the computations described in Section 7.1 of \cite{kopp1}. The conjectural Stark unit is $\exp(Z_I'(0)) \approx 3.89086171394307925533764395962$. This number is appears to be the root of the polynomial
\begin{align}
x^8 &- (8 + 5\sqrt{3})x^7 + (53 + 30\sqrt{3})x^6 - (156 + 90\sqrt{3})x^5 + (225 + 130\sqrt{3})x^4\nn\\ 
 &- (156 + 90\sqrt{3})x^3 + (53 + 30\sqrt{3})x^2 - (8 + 5\sqrt{3})x + 1,
\end{align}
which we have verified lies in the appropriate class field.

\section{Acknowledgements}

This research was partially supported by National Science Foundation (USA) grants DMS-1401224, DMS-1701576, and DMS-1045119, and by the Heilbronn Institute for Mathematical Research (UK).

This paper incorporates material from the author's PhD thesis \cite{koppthesis}. Thank you to Jeffrey C. Lagarias for advising my PhD and for many helpful conversations about the content of this paper. Thank you to Marcus Appleby, Jeffrey C. Lagarias, and Kartik Prasanna for helpful comments and corrections.

\section{Conflict of interest statement}

On behalf of all authors, the corresponding author states that there is no conflict of interest.

\bibliographystyle{plain}
\bibliography{references}

\end{document}